\documentclass[11pt,draft]{amsart}
\usepackage{mathrsfs}
\usepackage{amssymb}
\usepackage{amsmath}
\usepackage{amsfonts}
\usepackage{amsfonts,enumerate}
\usepackage{pifont}
\usepackage{enumerate}
\usepackage{graphics}
\usepackage{verbatim}
\usepackage{amsthm}

\begin{document}

\numberwithin{equation}{section}
\newtheorem{theorem}{Theorem}[section]
\newtheorem{corollary}{Corollary}[section]
\newtheorem{definition}{Definition}[section]
\newtheorem{lemma}{Lemma}[section]
\newtheorem{proposition}{Proposition}[section]
\newtheorem{remark}{Remark}[section]
\newtheorem{example}{Example}[section]
\newtheorem{claim}{Claim}[section]
\newtheorem{problem}{Problem}[section]

\newcommand{\real}{{\mathbb R}}
\newcommand{\nat}{{\mathbb N}}
\newcommand{\ent}{{\mathbb Z}}
\newcommand{\com}{{\mathbb C}}
\newcommand{\un}{1\mkern -4mu{\textrm l}}
\newcommand{\FF}{{\mathbb F}}
\newcommand{\T}{{\mathbb T}}
\newcommand{\A}{{\mathcal A}}
\newcommand{\B}{{\mathcal B}}
\newcommand{\E}{{\mathcal E}}
\newcommand{\F}{{\mathcal F}}
\newcommand{\G}{{\mathcal G}}
\newcommand{\M}{{\mathcal M}}
\newcommand{\N}{{\mathcal N}}
\newcommand{\R}{{\mathcal R}}
\newcommand{\U}{{\mathcal U}}
\newcommand{\BMO}{{\mathcal {BMO}}}
\newcommand{\g}{\gamma}
\newcommand{\Ga}{\Gamma}
\newcommand{\D}{\Delta}
\newcommand{\e}{\varepsilon}
\newcommand{\f}{\varphi}
\newcommand{\La}{\Lambda}
\newcommand{\s}{\sigma}
\newcommand{\Tr}{\mbox{\rm Tr}}
\newcommand{\tr}{\mbox{\rm tr}}
\newcommand{\ot}{\otimes}
\newcommand{\op}{\oplus}
\newcommand{\8}{\infty}
\newcommand{\el}{\ell}
\newcommand{\pa}{\partial}
\newcommand{\la}{\langle}
\newcommand{\ra}{\rangle}
\newcommand{\wt}{\widetilde}
\newcommand{\wh}{\widehat}
\newcommand{\n}{\noindent}
\newcommand{\bv}{\bigvee}
\newcommand{\w}{\wedge}
\newcommand{\bw}{\bigwedge}
\newcommand{\pf}{\noindent{\it Proof.~~}}
\newcommand{\ep}{\varepsilon}
\newcommand{\Om}{\Omega}
\renewcommand{\H}{{\mathcal H}}
\renewcommand{\P}{{\mathcal P}}
\renewcommand{\L}{{\mathcal L}}
\renewcommand{\S}{{\mathcal S}}
\renewcommand{\a}{\alpha}
\renewcommand{\b}{\beta}
\renewcommand{\d}{\delta}
\renewcommand{\th}{\theta}
\renewcommand{\l}{\lambda}
\renewcommand{\e}{\varepsilon}
\renewcommand{\O}{{\Omega}}
\renewcommand{\o}{{\omega}}
\renewcommand{\t}{\tau}
\renewcommand{\v}{\vee}
\newcommand{\h}{\mathsf{h}}
\newcommand{\at}{\mathrm{at}}
\newcommand{\cqd}{\hfill$\Box$}
\newcommand{\MO}{\mathcal{MO}}
\newcommand{\cond}{\mathrm{cond}}
\newcommand{\bmo}{\mathsf{bmo}}
\newcommand{\mo}{\mathsf{mo}}

\newcommand{\be}{\begin{eqnarray*}}
\newcommand{\ee}{\end{eqnarray*}}
\newcommand{\beq}{\begin{equation}}
\newcommand{\eeq}{\end{equation}}
\newcommand{\beqn}{\begin{equation*}}
\newcommand{\eeqn}{\end{equation*}}
\newcommand{\bs}{\begin{split}}
\newcommand{\es}{\end{split}}

\title{Real-variable characterizations\\ of Bergman spaces}


\author{Zeqian Chen}

\address{Wuhan Institute of Physics and Mathematics, Chinese
Academy of Sciences, 30 West District, Xiao-Hong-Shan, Wuhan 430071, China}

\author{Wei Ouyang}

\address{Wuhan Institute of Physics and Mathematics, Chinese
Academy of Sciences, 30 West District, Xiao-Hong-Shan, Wuhan 430071, China
and Graduate School, Chinese Academy of Sciences, Wuhan 430071, China}

\date{}
\maketitle

\markboth{Z. Chen and W.Ouyang}
{Bergman spaces}

\begin{abstract}
In this paper, we give a survey of results obtained recently by the present authors on real-variable characterizations of Bergman spaces, which are closely related to maximal and area integral functions in terms of the Bergman metric. In particular, we give a new proof of those results concerning area integral characterizations through using the method of vector-valued Calder\'{o}n-Zygmund operators to handle Bergman singular integral operators on the complex ball. The proofs involve some sharp estimates of the Bergman kernel function and Bergman metric.
\end{abstract}


\section{Introduction}\label{intro}

There is a mature and powerful real variable Hardy space theory which has distilled some of the essential oscillation and cancellation
behavior of holomorphic functions and then found that behavior ubiquitous. A good introduction to that is \cite{CW}; a more recent and fuller account is in \cite{AB, GL, Stein1993} and references therein. However, the real-variable theory of the Bergman space is less well developed, even in the case of the unit disc (cf. \cite{DS2004}).

Recently, in \cite{CO1} the present authors established real-variable type maximal and area integral characterizations of Bergman spaces in the unit ball of $\mathbb{C}^n.$ The characterizations are in terms of maximal functions and area functions on Bergman balls involving the radial derivative, the complex gradient, and the invariant gradient. Subsequently, in \cite{CO2} we introduced a family of holomorphic spaces of tent type in the unit ball of $\mathbb{C}^n$ and showed that those spaces coincide with Bergman spaces. Moreover, the characterizations extend to cover Besov-Sobolev spaces. A special case of this is a characterization of $\mathrm{H}^p$ spaces involving only area functions on Bergman balls.

We remark that the first real-variable characterization of the Bergman spaces was presented by Coifman and Weiss in 1970's. Recall that
\be
\varrho (z, w) = \left \{\begin{split}
& \big | |z| - |w| \big | + \Big | 1 - \frac{1}{|z| |w|} \langle z, w \rangle \Big |,\quad \text{if}\; z, w \in \mathbb{B}_n \backslash \{0\},\\
& |z| + |w|,\quad \text{otherwise}
\end{split} \right.
\ee
is a pseudo-metric on $\mathbb{B}_n$ and $(\mathbb{B}_n, \varrho, d v_{\alpha})$ is a homogeneous space. By their theory of harmonic analysis on homogeneous spaces, Coifman and Weiss \cite{CW} can use $\varrho$ to obtain a real-variable atomic decomposition for Bergman spaces. However, since the Bergman metric $\beta$ underlies the complex geometric structure of the unit ball of $\mathbb{C}^n,$ one would prefer to real-variable characterizations of the Bergman spaces in terms of $\beta.$ Clearly, the results obtained in \cite{CO1, CO2} are such a characterization.

In this paper, we will give a detailed survey of results obtained in \cite{CO1, CO2}. Moreover, we will give a new proof of those results concerning area integral characterizations through using the method of vector-valued Calder\'{o}n-Zygmund operator theory to handle Bergman singular integral operators on the complex ball. This paper is organized as follows. In Section \ref{Bergmanspace}, some notations and a number of auxiliary (and mostly elementary) facts about the Bergman kernel functions are presented. In Section \ref{AtomDecomp}, we will discuss real-variable type atomic decomposition of Bergman spaces. In particular, we will present the atomic decomposition of Bergman spaces with respect to Carleson tubes that was obtained in \cite{CO1}. Section \ref{RVC} is devoted to present maximal and area integral function characterizations of Bergman spaces. Finally, in Section \ref{BergmanOperator}, we will give a new proof of those results concerning the area integral characterizations obtained in \cite{CO1, CO2} using the argument of Calder\'{o}n-Zygmund operator theory through introducing Bergman singular integral operators on the complex ball.

In what follows, $C$ always denotes a constant depending (possibly) on $n, q, p, \g$ or $\alpha$ but not on $f,$ which may be
different in different places. For two nonnegative (possibly infinite) quantities $X$ and $Y,$ by $X \lesssim Y$ we mean that there
exists a constant $C>0$ such that $ X \leq C Y.$ We denote by $X \thickapprox Y$ when $X \lesssim Y$ and $Y \lesssim X.$ Any notation and terminology not otherwise explained, are as used in \cite{Z} for spaces of holomorphic functions in the unit ball of $\mathbb{C}^n.$

\section{Bergman spaces}\label{Bergmanspace}

Let $\mathbb{C}$ denote the set of complex numbers. Throughout the paper we fix a positive integer $n,$ and let ${\mathbb{B}_n}$ denote the open unit ball in $\mathbb{C}^n.$ The boundary of $\mathbb{B}_n$ will be denoted by $\mathbb{S}_n$ and is called the unit sphere in $\mathbb{C}^n.$ Also, we denote by $\overline{\mathbb{B}}_n$ the closed unit ball, i.e., $\overline{\mathbb{B}}_n= \{z \in \mathbb{C}^n:\; |z| \le 1 \} = \mathbb{B}_n \cup \mathbb{S}_n.$ The automorphism group of $\mathbb{B}^n,$ denoted by $\mathrm{Aut} (\mathbb{B}^n),$ consists of all bi-holomorphic mappings of $\mathbb{B}^n.$ Traditionally, bi-holomorphic mappings are also called automorphisms.

For $\alpha$ $\in$ $\mathbb{R}$, the weighted Lebesgue measure $dv_{\alpha}$ on $\mathbb{B}_n$ is defined by
\be
dv_{\alpha}(z)=c_{\alpha}(1-|z|^2)^{\alpha}dv(z)
\ee
where $c_{\alpha}=1$ for $\alpha\le-1$ and $c_{\alpha}=\Gamma(n+\alpha+1)/ [n!\Gamma(\alpha+1)]$ if $\alpha>-1$, which is a normalizing constant so that $dv_{\alpha}$
is a probability measure on $\mathbb{B}_n.$ In the case of $\alpha=-(n+1)$ we denote the resulting measure by
\be
d\tau(z)=\frac{d v (z)}{(1-|z|^2)^{n+1}},
\ee
and call it the invariant measure on $\mathbb{B}^n,$ since $d\tau=d\tau\circ\varphi$ for any automorphism $\varphi$ of $\mathbb{B}^n.$

Recall that for $\alpha > -1$ and $p>0$ the (weighted) Lebesgue space $L^p_{\alpha} (\mathbb{B}_n)$ (or, $L^p_{\alpha}$ in short) consists
of measurable (complex) functions $f$ on $\mathbb{B}_n$ with
\be
\|f\|_{p,\,\alpha}=\left ( \int_{\mathbb{B}_n}|f(z)|^pdv_{\alpha}(z) \right )^{\frac{1}{p}}<\infty.
\ee
The (weighted) Bergman space $\mathcal{A}^p_{\alpha}$ is then defined as
\be
\mathcal{A}^p_{\alpha} = \mathcal{H} (\mathbb{B}_n) \cap L^p_{\alpha},
\ee
where $\mathcal{H} (\mathbb{B}_n)$ is the space of all holomorphic functions in $\mathbb{B}_n.$ When $\alpha =0$ we simply write $\A^p$ for $\A^p_0.$ These are the usual Bergman spaces. Note that for $1 \le p < \8,$ $\mathcal{A}^p_{\alpha}$ is a Banach space under the norm $\|\ \|_{p,\,\alpha}.$ If $0 < p <1,$ the space $\mathcal{A}^p_{\alpha}$ is a quasi-Banach space with $p$-norm $\| f \|^p_{p, \alpha}.$

Recall that the dual space of $\mathcal{A}^1_{\alpha}$ is the Bloch space $\mathcal{B}$ defined as follows
(we refer to \cite {Z} for details). The Bloch space $\mathcal{B}$ of $\mathbb{B}_n$ is defined to be the space of holomorphic  functions $f$ in ${\mathbb{B}_n}$ such that
\be
\|f\|_{\mathcal B}=\sup\{\,|\widetilde{\nabla}f(z)|:z\in{\mathbb{B}_n}\}<\infty.
\ee
$\|\ \|_{\mathcal B}$ is a semi-norm on $\mathcal{B}.$ $\mathcal{B}$ becomes a Banach space with the following norm
$$\|f\| = |f(0)| + \|f\|_{\mathcal B}.$$
It is known that the Banach dual of $\mathcal A^1_{\alpha}$ can be identified with $\mathcal B$ (with equivalent norms) under the integral pairing
\be
\langle f,g \rangle_{\alpha}= \lim_{r \to 1^-}\int_{\mathbb{B}_n} f(r z) \overline{g(z)} d v_{\alpha}(z),\quad f \in{\mathcal A^1_{\alpha}},\; g \in \mathcal{B}.
\ee
(e.g., see Theorem 3.17 in \cite{Z}.)

We define the so-called generalized Bergman spaces as follows (e.g., \cite{ZZ}). For $0 < p < \8$ and $-\8 < \alpha < \8$ we fix a nonnegative integer $k$ with $pk + \alpha > -1$ and define $\mathcal{A}^p_{\alpha}$ as the space of all $f \in \mathcal{H} (\mathbb{B}_n)$ such that $(1-|z|^2)^k \mathcal{R}^k f \in L^p (\mathbb{B}_n, d v_{\alpha}).$ One then easily observes that $\mathcal{A}^p_{\alpha}$ is independent of the choice of $k$ and consistent with the traditional definition when $\alpha > -1.$ Let $N$ be the smallest nonnegative integer
such that $pN + \alpha > -1$ and define
\beq\label{eq:BergmanSpaceNorm}
\| f \|_{p, \alpha} = | f(0)| + \left ( \int_{\mathbb{B}_n} (1-|z|^2)^{pN} | \mathcal{R}^N f (z) |^p d v_{\alpha} (z) \right )^{\frac{1}{p}},\quad f \in \mathcal{A}^p_{\alpha}\;.
\eeq
Equipped with \eqref{eq:BergmanSpaceNorm}, $\mathcal{A}^p_{\alpha}$ becomes a Banach space when $p \ge 1$ and a quasi-Banach space for $0 < p < 1.$

Note that the family of the generalized Bergman spaces $\mathcal{A}^p_{\alpha}$ covers most of the spaces of holomorphic functions in the unit ball of $\mathbb{C}^n,$ which has been extensively studied before in the literature under different names. For example, $B^s_p = \mathcal{A}^p_{\alpha}$ with $\alpha = - (ps+1),$ where $B^s_p$ is the classical diagonal Besov space consisting of holomorphic functions $f$ in $\mathbb{B}_n$ such that $(1-|z|^2)^{k-s} \mathcal{R}^k f$ belongs to $L^p(\mathbb{B}_n, d v_{-1})$ with $k$ being any positive integer greater than $s.$ It is clear that $\mathcal{A}^p_{\alpha} = B^s_p$ with $s = - (\alpha + 1)/p.$ Thus the generalized Bergman spaces $\mathcal{A}^p_{\alpha}$ are exactly the diagonal Besov spaces. On the other hand, if $k$ is a positive integer, $p$ is positive, and $\beta$ is real, then there is the Sobolev space $W^p_{k,\beta}$ consisting of holomorphic functions $f$ in $\mathbb{B}_n$ such that the partial derivatives of $f$ of order up to $k$ all belong to $L^p ( \mathbb{B}_n, d v_{\beta})$ (cf. \cite{AC, A, BB}). It is easy to see that these holomorphic Sobolev spaces are in the scale of the generalized Bergman spaces, i.e., $W^p_{k,\beta} = \mathcal{A}^p_{\alpha}$ with $\alpha = - (pk - \beta +1)$ (e.g., see \cite{ZZ} for an overview). We refer to Arcozzi-Rochberg-Sawyer \cite{ARS2006, ARS2008}, Tchoundja \cite{T2008} and Volberg-Wick \cite{VW} for some recent results on such Besov spaces and more references.

Recall that $D(z,\gamma)$ denotes the Bergman metric ball at $z$
\be
D(z, \gamma) = \{w \in \mathbb{B}_n\;: \; \beta (z, w) < \g \}
\ee
with $\gamma >0,$ where $\beta$ is the Bergman metric on $\mathbb{B}_n.$ It is known that
\be
\beta (z, w) = \frac{1}{2} \log \frac{1 + | \varphi_z (w)|}{1 - | \varphi_z (w)|},\quad z, w \in \mathbb{B}_n,
\ee
whereafter $\varphi _z$ is the bijective holomorphic mapping in $\mathbb{B}_n,$ which satisfies $\varphi _z (0)=z$, $\varphi _z(z)=0$ and
$\varphi _z\circ\varphi _z = id.$ If $\mathbb{B}_n$ is equipped with the Bergman metric $\beta,$ then $\mathbb{B}_n$ is a separable metric space. We shall call
$\mathbb{B}_n$ a separable metric space instead of $(\mathbb{B}_n,\beta).$

For reader's convenience we collect some elementary facts on the Bergman metric and holomorphic functions in the unit ball of $\mathbb{C}^n.$

\begin{lemma}\label{le:Estimation1} (cf. Lemma 1.24 in \cite{Z})
For any real $\alpha$ and positive $\g$ there exist constant $C_{\gamma}$ such that
\be
C_{\gamma}^{-1}(1-|z|^2)^{n+1+\alpha}\le v_{\alpha}(D(z,\g))\le C_{\gamma}(1-|z|^2)^{n+1+\alpha}
\ee
for all $z\in\mathbb{B}_n$.
\end{lemma}

\begin{lemma}\label{le:Estimation2} (cf. Lemma 2.20 in \cite{Z})
For each $\gamma>0,$
\be
1-|a|^2 \approx 1-|z|^2 \approx |1-\langle a,z\rangle|
\ee
for all $a$ in $\mathbb{B}_n$ with $z \in D(a, \gamma).$
\end{lemma}


\begin{lemma}\label{le:Estimation4} (cf. Lemma 2.27 in \cite{Z})
For each $\gamma>0,$
\be
|1-\langle z,u\rangle| \approx |1-\langle z,v\rangle|
\ee
for all $z$ in $\bar{\mathbb{B}}_n$ and $u,v$ in $\mathbb{B}_n$ with $\beta(u,v)<\gamma.$
\end{lemma}

\section{Atomic decomposition}\label{AtomDecomp}

We first recall the following ``complex-variable" atomic decomposition for Bergman spaces due to Coifman and Rochberg \cite{CR} (see also \cite{Z}, Theorem 2.30).

\begin{theorem}\label{th:complexatomicdecomp}
Suppose $p>0, \alpha>-1,$ and $b> n \max \{1, 1/p \} + (\alpha+1)/p.$ Then there exists a sequence $\{a_k\}$ in $\mathbb{B}_n$ such that
$\mathcal{A}^p_{\alpha}$ consists exactly of functions of the form
\beq\label{eq:atomdecomp}
f(z)=\sum^{\infty}_{k=1}c_{k}\frac{(1-|a_k|^2)^{(pb-n-1-\alpha)/p}}{(1-\langle z,a_{k}\rangle)^b},\quad  z\in\mathbb{B}_n,
\eeq
where $\{c_k\}$ belongs to the sequence space $\el^p$ and the series converges in the norm topology of $\mathcal{A}^p_{\alpha}.$ Moreover,
\be
\int_{\mathbb{B}_n}|f(z)|^pdv_{\alpha}(z) \approx  \inf \Big \{ \sum_{k}|c_{k}|^p \Big \},
\ee
where the infimum runs over all the above decompositions.
\end{theorem}

By Theorem \ref{th:complexatomicdecomp} we conclude that for any $\alpha >-1,$ $\mathcal{A}^p_{\alpha}$ as a Banach space is isomorphic to $\el^p$ for every $1 \le p < \8.$

Now we turn to the real-variable atomic decomposition of Bergman spaces. To this end, we need some more notations as follows.

For any $\zeta\in{\mathbb{S}_n}$ and $r>0,$ the set
$$Q_r(\zeta)=\{z\in \mathbb{B}_n: d(z,\zeta)<r\}$$
is called a Carleson tube with respect to the nonisotropic metric $d.$ We usually write $Q = Q_r(\zeta)$ in short.

As usual, we define the atoms with respect to the Carleson tube as follows: for $1 <q<\infty,$ $a \in L^q (\mathbb{B}_n, d v_{\alpha})$ is said to be a $(1, q)_{\alpha}$-atom if there is a Carleson tube $Q$ such that
\begin{enumerate}[{\rm (1)}]

\item $a$ is supported in $Q;$

\item $\| a \|_{L^q (\mathbb{B}_n, d v_{\alpha})} \le v_{\alpha} (Q)^{\frac{1}{q}-1};$

\item $\int_{\mathbb{B}_n}a(z)\,dv_{\alpha}(z) = 0.$

\end{enumerate}
The constant function $1$ is also considered to be a $(1, q)_{\alpha}$-atom.

Note that for any $(1, q)_{\alpha}$-atom $a,$
\be
\| a \|_{1, \alpha} = \int_Q | a | d v_{\alpha} \le v_{\alpha} (Q)^{1 - 1/q} \| a \|_{q, \alpha} \le 1.
\ee
Then, we define $\mathcal{A}^{1,q}_{\alpha}$ as the space of all $f \in \mathcal{A}^1_{\alpha}$ which admits a decomposition
\be
f=\sum_i \lambda_i P_{\alpha} a_i\quad \text{and}\quad \sum_i |\lambda_i | \le C_q \|f\|_{1,\,\alpha},
\ee
where for each $i,$ $a_i$ is an $(1,q)_{\alpha}$-atom and $\lambda_i \in \mathbb{C}$ so that $\sum_i | \lambda_i | < \8.$ We equip this space with the norm
\be
\|f\|_{\mathcal{A}^{1,q}_{\alpha}} = \inf \Big \{ \sum_i |\lambda_i |:\; f=\sum_i \lambda_i P_{\alpha} a_i \Big \}
\ee
where the infimum is taken over all decompositions of $f$ described above.

It is easy to see that  $\mathcal{A}^{1,q}_{\alpha}$ is a Banach space.


\begin{theorem}\label{th:AtDecomp}
Let $1< q<\infty$ and $\alpha>-1.$ For every $f \in{\mathcal A^1_{\alpha}}$ there exist a sequence $\{a_i\}$ of $(1, q)_{\alpha}$-atoms and a sequence $\{\lambda_i\}$ of complex numbers such that
\beq \label{eq:AtDecomp}
f=\sum_i \lambda_i P_{\alpha} a_i\quad \text{and}\quad \sum_i |\lambda_i | \le C_q \|f\|_{1,\,\alpha}.
\eeq
Moreover,
\be
\|f\|_{1,\,\alpha} \approx \inf \sum_i |\lambda_i |
\ee
where the infimum is taken over all decompositions of $f$ described above and $`` \approx "$ depends only on $\alpha$ and $q.$
\end{theorem}

Theorem \ref{th:AtDecomp} is proved in \cite{CO1} via duality.

\begin{remark}\label{rk:atomdecomp-p<1}
One would like to expect that when $0 < p < 1,$ $\A^p_{\alpha}$ also admits an atomic decomposition in terms of atoms with respect to Carleson tubes. However, the proof of Theorem \ref{th:AtDecomp} via duality cannot be extended to the case $0< p<1.$ At the time of this writing, this problem is entirely open.

As mentioned in Introduction, by their theory of harmonic analysis on homogeneous spaces, Coifman and Weiss \cite{CW} have obtained a real-variable atomic decomposition in terms of $\varrho$ for Bergman spaces in the case $0 < p \le 1.$
\end{remark}

\section{Real-variable characterizations}\label{RVC}

\subsection{Maximal functions}\label{maximalfunct}

As is well known, maximal functions play a crucial role in the real-variable theory of Hardy spaces (cf. \cite{Stein1993}). In \cite{CO1}, the authors established a maximal-function characterization for the Bergman spaces. To this end, we define for each $\g > 0$ and $f \in \mathcal{H} (\mathbb{B}_n):$
\beq\label{eq:MaximalFunct}
(\mathrm{M}_{\g} f) (z) = \sup_{w \in D(z, \g)} | f (w)|,\; \forall z \in \mathbb{B}_n.
\eeq
The following result is proved in \cite{CO1}.

\begin{theorem}\label{th:MaxialCharat}
Suppose $\gamma >0$ and $\alpha > -1.$ Let $0< p < \8.$ Then for any $f \in \mathcal{H} (\mathbb{B}_n),$ $f \in \mathcal{A}^p_{\alpha}$ if and only if $\mathrm{M}_{\g} f \in L^p (\mathbb{B}_n, d v_{\alpha}).$ Moreover,
\beq\label{eq:MaximalFunctNorm}
\| f \|_{p,\alpha} \approx \| \mathrm{M}_{\g} f \|_{p, \alpha},
\eeq
where ``$\approx$" depends only on $\gamma, \alpha, p,$ and $n.$
\end{theorem}

The norm appearing on the right-hand side of \eqref{eq:MaximalFunctNorm} can be viewed an analogue of the so-called nontangential maximal function in Hardy spaces. The proof of Theorem \ref{th:MaxialCharat} is fairly elementary, using some basic facts and estimates on the Bergman balls.

\begin{corollary}\label{cor:MaxialCharat}
Suppose $\gamma >0$ and $\alpha \in \mathbb{R}.$ Let $0< p < \8$ and $k$ be a nonnegative integer such that $pk + \alpha > -1.$ Then for any $f \in \mathcal{H} (\mathbb{B}_n),$ $f \in \mathcal{A}^p_{\alpha}$ if and only if $\mathrm{M}_{\g} ( \mathcal{R}^k f ) \in L^p (\mathbb{B}_n, d v_{\alpha}),$
where
\beq\label{eq:kMaximalFunct}
\mathrm{M}_{\g} ( \mathcal{R}^k f ) (z) = \sup_{w \in D (z, \gamma)} | (1-|w|^2)^k \mathcal{R}^k f (w) |,\quad z \in \mathbb{B}_n.
\eeq
Moreover,
\beq\label{eq:kMaximalFunctNorm}
\| f \|_{p,\alpha} \approx | f(0)| + \| \mathrm{M}_{\g} ( \mathcal{R}^k f ) \|_{p, \alpha},
\eeq
where ``$\approx$" depends only on $\gamma, \alpha, p, k,$ and $n.$
\end{corollary}

To prove Corollary \ref{cor:MaxialCharat}, one merely notices that $f \in \mathcal{A}^p_{\alpha}$ if and only if $\mathcal{R}^k f \in L^p (\mathbb{B}_n, d v_{\alpha + pk})$ and applies Theorem \ref{th:MaxialCharat} to $\mathcal{R}^k f$ with the help of Lemma \ref{le:Estimation2}.

\subsection{Area integral functions}\label{areafunct}

In order to state the real-variable area integral characterizations of the Bergman spaces, we require some more notation. For any $f \in \mathcal{H} ( \mathbb{B}_n)$ and $z = (z_1, \ldots, z_n) \in \mathbb{B}_n$ we define
\be
\mathcal{R} f (z) = \sum^n_{k=1} z_k \frac{\partial f (z)}{\partial z_k}
\ee
and call it the radial derivative of $f$ at $z.$ The complex and invariant gradients of $f$ at $z$ are respectively defined as
\be
\nabla f(z) = \Big ( \frac{\partial f (z) }{\partial z_1},\ldots, \frac{\partial f (z) }{\partial z_n} \Big )\; \text{and}\; \widetilde{\nabla} f (z) = \nabla(f \circ \varphi_z)(0).
\ee

Now, for fixed $\gamma >0$ and $1 < q < \8,$ we define for each $f \in \mathcal{H} (\mathbb{B}_n)$ and $z \in \mathbb{B}_n:$
\begin{enumerate}[{\rm (1)}]

\item The radial area function
\be
A_{\gamma}^{(q)} (\mathcal{R} f) (z) = \left ( \int_{D(z,\gamma)}  | (1-|w|^2) \mathcal{R} f (w) |^q d \tau (w) \right )^{\frac{1}{q}}.
\ee

\item The complex gradient area function
\be
A_{\gamma}^{(q)} (\nabla f) (z) = \left ( \int_{D(z, \gamma)} | (1-|w|^2) \nabla f (w) |^q d \tau (w) \right )^{\frac{1}{q}}.
\ee

\item The invariant gradient area function
\be
A_{\gamma}^{(q)} (\tilde{\nabla} f) (z) = \left ( \int_{D(z, \gamma)} | \tilde{\nabla} f (w) |^q d \tau (w) \right )^{\frac{1}{q}}.
\ee

\end{enumerate}

The following theorem is proved in \cite{CO1}.

\begin{theorem}\label{th:AreaCharat}
Suppose $\gamma > 0, 1 < q < \8,$ and $\alpha > -1.$ Let $0 < p < \8.$ Then, for any $f \in \mathcal{H} (\mathbb{B}_n)$ the following conditions are equivalent:
\begin{enumerate}[{\rm (a)}]

\item $f \in \mathcal{A}^p_{\alpha}.$

\item $A_{\gamma}^{(q)} (\mathcal{R} f)$ is in $L^p (\mathbb{B}_n, d v_{\alpha}).$

\item $A_{\gamma}^{(q)} (\nabla f)$ is in $L^p (\mathbb{B}_n, d v_{\alpha}).$

\item $A_{\gamma}^{(q)} (\tilde{\nabla} f)$ is in $L^p (\mathbb{B}_n, d v_{\alpha}).$

\end{enumerate}
Moreover, the quantities
\be
|f(0)| + \| A_{\gamma}^{(q)} (\mathcal{R} f) \|_{p, \alpha},\; |f(0)| + \| A_{\gamma}^{(q)} (\nabla f) \|_{p, \alpha}, \; |f(0)| + \| A_{\gamma}^{(q)} (\tilde{\nabla} f) \|_{p, \alpha},
\ee
are all comparable to $\| f \|_{p, \alpha},$ where the comparable constants depend only on $\gamma, q, \alpha, p,$ and $n.$
\end{theorem}

In particular, taking the equivalence of (a) and (b), one obtains
\be
\| f \|_{p, \alpha} \approx |f(0)| + \| A_{\gamma}^{(q)} (\mathcal{R} f) \|_{p, \alpha},
\ee
which looks tantalizingly simple. However, the authors know no simple proof of this fact even in the case of the usual Bergman space on the unit disc.

\begin{corollary}\label{cor:AreaCharat}
Suppose $\gamma >0, 1 < q < \8,$ and $\alpha \in \mathbb{R}.$ Let $0 < p < \8$ and $k$ be a nonnegative integer such that $pk + \alpha > -1.$ Then for any $f \in \mathcal{H} (\mathbb{B}_n),$ $f \in \mathcal{A}^p_{\alpha}$ if and only if $A_{\gamma}^{(q)} (\mathcal{R}^{k+1} f)$ is in $L^p (\mathbb{B}_n, d v_{\alpha}),$
where
\beq\label{eq:kareaFunct}
A_{\gamma}^{(q)} (\mathcal{R}^k f) (z) = \left ( \int_{D(z,\gamma)} \big | (1-|w|^2)^k \mathcal{R}^k f (w) \big |^q d \tau (w) \right )^{\frac{1}{q}}.
\eeq
Moreover,
\beq\label{eq:kAreaFunctNorm}
\| f \|_{p,\alpha} \approx | f(0)| + \| A_{\gamma}^{(q)} (\mathcal{R}^{k+1} f) \|_{p, \alpha},
\eeq
where ``$\approx$" depends only on $\gamma, q, \alpha, p, k,$ and $n.$
\end{corollary}

To prove Corollary \ref{cor:AreaCharat}, one merely notices that $f \in \mathcal{A}^p_{\alpha}$ if and only if $\mathcal{R}^k f \in L^p (\mathbb{B}_n, d v_{\alpha + pk})$ and applies Theorem \ref{th:AreaCharat} to $\mathcal{R}^k f$ with the help of Lemma \ref{le:Estimation2}.






\subsection{Tent spaces}\label{tentspace}

The basic functional used below is the one mapping functions in $\mathbb{B}_n$ to functions in $\mathbb{B}_n,$ given by
\beq\label{eq:TentFunct1}
A^{(q)}_{\gamma}(f)(z) = \left(\int_{D(z, \g)} | f (w)|^qd\tau(w)\right)^{\frac{1}{q}}
\eeq
if $1 < q < \infty,$ and
\beq\label{eq:TentFunct2}
A^{(\infty)}_{\gamma} (f) (z) = \sup_{w \in D(z, \g)} | f (w)|,\quad \mathrm{when} \ q=\infty.
\eeq
Then, the ``holomorphic space of tent type" $T^p_{q,\alpha}$ in $\mathbb{B}_n$ is defined as the
holomorphic functions $f$ in $\mathbb{B}_n$ so that $A^{(q)}_{\gamma}(f) \in L^p_{\alpha},$ when $0<p \le\infty$ and
$\alpha>-1,\; \gamma > 0, 1 < q \le \8.$ The corresponding classes are then equipped with a norm (or, quasi-norm) $\|f \|_{T^p_{q,\alpha}} = \|A^{(q)}_{\g}(f)\|_{p,\alpha}.$ This motivation aries form the tent spaces in $\mathbb{R}^n,$ which were introduced and developed by Coifman, Meyer and Stein in \cite{CMS1985}.

The case of $q=\infty$ and $0<p<\infty$ was studied in Section \ref{maximalfunct} (see \cite{CO1} for details). Actually, the resulting tent type spaces $T^{p}_{\infty,\alpha}$ is Bergman spaces $\mathcal{A}^p_{\alpha}.$ It is clear that $T^{\infty}_{q,\alpha}$ with $1<q<\infty$ is imbedded in Bloch space. On the other hand, $T^p_{q,\alpha}$ are Banach spaces when $p\ge 1.$

It is well known that the Hardy-Littlewood maximal function operator has played important role in harmonic analysis. To cater our estimates, we use two variants of the non-central Hardy-Littlewood maximal function operator acting on the weighted Lebesgue spaces $L^p_\alpha({\mathbb{B}_n})$, namely,
\beq\label{eq:MaxFunct2}
M^{(q)}_{\g} (f) (z) = \sup_{z \in D(w, \g)}\left(\frac{1}{v_{\alpha}(D(w,\g))}\int_{D(w,\g)}|f|^qdv_{\alpha}\right)^{\frac{1}{q}}
\eeq
for $0 < q < \8.$ We simply write $M_{\g} (f) (z): = M^{(1)}_{\g} (f) (z).$

The following result is proved in \cite{CO2}.

\begin{theorem}\label{th:tentspace}
Suppose $\gamma >0, 1 < q < \8,$ and $\alpha > -1.$ Let $0< p <\8.$ Then for any $f \in \mathcal{H} (\mathbb{B}_n),$ the following conditions are equivalent:
\begin{enumerate}[{\rm (1)}]

\item $f \in \A^p_{\alpha}.$

\item $A^{(q)}_{\g} (f)$ is in $L^p(\mathbb{B}_n, d v_{\alpha}).$

\item $M^{(q)}_{\g} (f)$ is in $L^p(\mathbb{B}_n, d v_{\alpha}).$

\end{enumerate}
Moreover,
\be
\| f \|_{\A^p_{\alpha}} \approx \| f \|_{T^p_{q, \alpha}} \approx \| M^{(q)}_{\g} (f) \|_{p, \alpha},
\ee
where the comparable constants depend only on $\gamma, q, \alpha, p,$ and $n.$
\end{theorem}

Note that the Bergman metric $\beta$ is non-doubling on $\mathbb{B}^n$ and so $(\mathbb{B}_n, \beta, d v_{\alpha})$ is a non-homogeneous space. The proof of the above theorem does involve some techniques of non-homogeneous harmonic analysis developed in \cite{NTV}.

\begin{corollary}\label{cor:AreaMaximalCharat}
Suppose $\gamma >0, 1 < q < \8,$ and $\alpha \in \mathbb{R}.$ Let $0 < p < \8$ and $k$ be a nonnegative integer such that $p k + \alpha > -1.$ Then for any $f \in \mathcal{H} (\mathbb{B}_n),$ $f \in \mathcal{A}^p_{\alpha}$ if and only if $A^{(q)}_{\gamma} (\mathcal{R}^k f)$ is in $L^p (\mathbb{B}_n, d v_{\alpha})$ if and only if $M^{(q)}_{\gamma} (\mathcal{R}^k f)$ is in $L^p (\mathbb{B}_n, d v_{\alpha}),$
where
\beq\label{eq:kareaFunct}
A^{(q)}_{\gamma} (\mathcal{R}^k f) (z) = \left ( \int_{D(z,\gamma)} \big | (1-|w|^2)^k \mathcal{R}^k f (w) \big |^q d \tau (w) \right )^{\frac{1}{q}}
\eeq
and
\beq\label{eq:kmaximalFunct}
M^{(q)}_{\gamma} (\mathcal{R}^k f) (z)  = \sup_{z \in D(w, \g)}\left ( \int_{D(w,\gamma)} \big | (1-|u|^2)^k \mathcal{R}^k f (u) \big |^q  \frac{d v_{\alpha} (u)}{v_{\alpha} (D(w,\gamma))}\right )^{\frac{1}{q}}
\eeq

Moreover,
\beq\label{eq:kAreaFunctNorm}
\| f \|_{p,\alpha} \approx | f(0)| + \| A^{(q)}_{\gamma} (\mathcal{R}^k f) \|_{p, \alpha} \approx | f(0)| + \| M^{(q)}_{\gamma} (\mathcal{R}^k f) \|_{p, \alpha},
\eeq
where ``$\approx$" depends only on $\gamma, q, \alpha, p, k,$ and $n.$
\end{corollary}

To prove Corollary \ref{cor:AreaMaximalCharat}, one merely notices that $f \in \mathcal{A}^p_{\alpha}$ if and only if $\mathcal{R}^k f \in L^p (\mathbb{B}_n, d v_{\alpha + pk})$ and applies Theorem \ref{th:tentspace} to $\mathcal{R}^k f$ with the help of Lemma \ref{le:Estimation2}. When $\alpha > -1,$ we can take $k=1$ and then recover Theorem \ref{th:AreaCharat}.

As mentioned in Section \ref{Bergmanspace}, the family of the generalized Bergman spaces $\mathcal{A}^p_{\alpha}$ covers most of the spaces of holomorphic functions in the unit ball of $\mathbb{C}^n,$ such as the classical diagonal Besov space $B^s_p$ and the Sobolev space $W^p_{k,\beta}.$ In particular, $\mathcal{H}^p_s = \mathcal{A}^p_{\alpha}$ with $\alpha = -2 s -1,$ where $\mathcal{H}^p_s$ is the Hardy-Sobolev space defined as the set
\be
\left \{ f \in \mathcal{H} (\mathbb{B}_n ):\; \| f \|^p_{\mathcal{H}^p_s} = \sup_{0<r <1} \int_{\mathbb{S}_n} | (I + \mathcal{R} )^s f (r \zeta) |^p  d \sigma (\zeta) < \8 \right \}.
\ee
Here,
\be
(I + \mathcal{R} )^s f = \sum^{\8}_{k=0} (1 + k)^s f_k
\ee
if $f = \sum^{\8}_{k=0} f_k$ is the homogeneous expansion of $f.$ There are several real-variable characterizations of the Hardy-Sobolev spaces obtained by Ahern and Bruna \cite{AB} (see also \cite{A}). These characterizations are in terms of maximal and area functions on the admissible approach region
\be
D_{\alpha} (\eta) = \left \{ z \in \mathbb{B}_n:\; | 1- \langle z, \eta \rangle | < \frac{\alpha}{2} (1 - |z|^2) \right \},\quad \eta \in \mathbb{S}_n,\; \alpha >1.
\ee
Evidently, Corollary \ref{cor:AreaMaximalCharat} present new real-variable descriptions of the Hardy-Sobolev spaces in terms of the Bergman metric. A special case of this is a characterization of the usual Hardy space $\mathcal{H}^p= \mathcal{A}^p_{-1}$ itself.

\section{Bergman integral operators}\label{BergmanOperator}

\subsection{Vector-valued kernels and Calder\'{o}n-Zygmund operators on homogeneous spaces}

Recall that a quasimetric on a set $X$ is a map $\rho$ from $X \times X$ to $[0, \8)$ such that
\begin{enumerate}[{\rm (1)}]

\item $\rho (x, y) =0$ if and only if $x =y;$

\item $\rho (x, y) = \rho (y, x);$

\item there exists a positive constant $C \ge 1$ such that
\be
\rho (x, y) \le C [ \rho (x, z) + \rho (z, y)],\quad \forall x,y,z \in X,
\ee
(the quasi-triangular inequality).

\end{enumerate}
For any $x \in X$ and $r >0,$ the set $B(x,r) = \{ y \in X: \; \rho (x,y) < r \}$ is called a $\rho$-ball of center $x$ and radius $r.$

A space of homogeneous type is a topological space $X$ endowed with a quasimetrc $\rho$ and a Borel measure $\mu$ such that
\begin{enumerate}[{\rm (a)}]

\item for each $x \in X,$ the balls $B(x,r)$ form a basis of open neighborhoods of $x$ and, also, $\mu (B(x,r))>0$ whenever $r >0;$

\item (doubling property) there exists a constant $C>0$ such that for each $x \in X$ and $r>0,$ one has
\be
\mu (B(x,2r)) \le C \mu (B(x,r)).
\ee

\end{enumerate}
$(X, \rho, \mu)$ is called a space of homogeneous type or simply a homogeneous space. We will usually abusively call $X$ a homogeneous space instead of $(X, \rho, \mu).$ We refer to \cite{CW, Stein1993} for details on harmonic analysis on homogeneous spaces.

Let $E$ be a Banach space. Let $L^p (\mu, E)$ be the usual Bochner-Lebesgue space for $1 \le p \le \8,$ and let $L^{1, \8} (\mu, E)$ be defined by
\be
L^{1,\8} (\mu, E): = \big \{ f: \; X \mapsto E | f\; \text{is strongly measurable such that}\; \| f \|_{L^{1, \8}(\mu, E)} < \8 \big \},
\ee
where $\| f \|_{L^{1, \8}(\mu, E)} : = \sup_{t >0} t \mu \big ( \{ x \in X : \|f(x)\|_E > t \} \big ).$ Note that $\| f \|_{L^{1, \8}} $ is not actually a norm in the sense that it does not satisfy the triangle inequality. However, we still have
\be
\| c f \|_{L^{1, \8}(\mu, E)} = |c| \| f \|_{L^{1, \8}(\mu, E)} \; \text{and}\; \| f + g \|_{L^{1, \8}(\mu, E)}  \le 2 ( \| f \|_{L^{1, \8}(\mu, E)} + \| g \|_{L^{1, \8}(\mu, E)} )
\ee
for every $c \in \mathbb{C}$ and $f, g \in L^{1,\8} (\mu, E).$

If $E= \mathbb{C}$ we simply write $L^p (\mu, E) = L^p (\mu)$ and $L^{1, \8} (\mu, E) = L^{1, \8} (\mu).$

Fix $m>0$ (not necessarily an integer). Define $\triangle = \{ (x,x):\; x \in X \}.$ A vector-valued $m$-dimensional Calder\'{o}n-Zygmund kernel with respect to $\rho$ is a continuous mapping $K : \; X \times X \backslash \triangle \mapsto E$ for which we have
\begin{enumerate}[{\rm (a)}]

\item there exists a constant $C_1>0$ such that
\be
\| K (x,y) \|_E \le \frac{C_1}{\rho (x,y)},\quad \forall x,y \in X \times X \backslash \triangle;
\ee

\item there exist constants $0< \epsilon \le 1$ and $C_2, C_3 >0$ such that
\be
\| K (x,y) - K(x', y) \|_E + \|K(y, x) - K(y, x') \|_E \le C_2 \frac{\rho (x,x')^{\epsilon}}{\rho (x, y)^{m+ \epsilon}}
\ee
whenever $x, x',y \in X$ and $\rho (x, x') \le C_3 \rho (x,y).$

\end{enumerate}

Given a vector-valued $m$-dimensional Calder\'{o}n-Zygmund kernel $K,$ we can define (at least formally) a Calder\'{o}n-Zygmund singular integral operator associated with this kernel by
\be
T f (x) = \int_X K(x,y) f (y) d \mu (x).
\ee


\begin{proposition}\label{prop:VCZO}
Let $E$ be a Banach space. If a Calder\'{o}n-Zygmund singular integral operator $T$ is bounded from $L^q (\mu)$ into $L^q (\mu, E)$ for some fixed $1 \le q < \8,$ then $T$ can be extended to an operator on $L^p (\mu)$ for every $1 \le p < \8$ such that
\begin{enumerate}[{\rm (a)}]

\item $T$ is $L^p$-bounded for every $1 < p < \8,$ i.e., $\| Tf \|_{L^p (\mu, E)} \le C_p \| f \|_{L^p(\mu)};$

\item $T$ is of weak type $(1, 1),$ i.e., $\| T f \|_{L^{1, \8}(\mu, E)} \le C \| f \|_{L^1 (\mu)}$ for all $f \in L^1(\mu);$

\item $T$ is bounded from $L^{\8} (\mu)$ into $\mathrm{BMO} (X, \rho, \mu; E);$

\item $T$ is bounded from $\mathrm{H}^1 (X, \rho, \mu)$ into $L^1 (\mu).$

\end{enumerate}
\end{proposition}

$\mathrm{H}^1 (X, \rho, \mu)$ and $\mathrm{BMO} (X, \rho, \mu; E)$ can be defined in a natural way, see \cite{CW} for the details. This result must be known for experts in the field of vector-valued harmonic analysis, and the proof can be obtained by merely modifying the proof of Theorem V.3.4 in \cite{GR1985}.

\subsection{Bergman singular integral operators}

We now turn our attention to the special case of the unit ball $\mathbb{B}_n.$ Recall that
\be
\varrho (z, w) = \left \{\begin{split}
& \big | |z| - |w| \big | + \Big | 1 - \frac{1}{|z| |w|} \langle z, w \rangle \Big |,\quad \text{if}\; z, w \in \mathbb{B}_n \backslash \{0\},\\
& |z| + |w|,\quad \text{otherwise}.
\end{split} \right.
\ee
It is know that $\varrho$ is a pseudo-metric on $\mathbb{B}_n$ and $(\mathbb{B}_n, \varrho, v_{\alpha})$ is a homogeneous space for $\alpha > -1$ (e.g., Lemma 2.10 in \cite{T2008}).

Let $E$ be a Banach space. Suppose $\alpha >-1.$ We are interested in vector-valued Bergman type integral operators on the unit ball $\overline{\mathbb{B}}_n$ in $\mathbb{C}^n.$ More precisely, we are interested in Bergman type integral operators whose kernels with values in $E$ satisfy the following estimates
\beq\label{eq:BergmanKernelEst1}
\|K (z, w)\|_E \le \frac{C}{\varrho (z, w)^{n+1 + \alpha}},\quad \forall (z, w) \in \mathbb{B}_n \times \mathbb{B}_n \setminus \{(\zeta, \zeta):\; \zeta \in \mathbb{B}_n \},
\eeq
and
\beq\label{eq:BergmanKernelEst2}
\|K (z, w ) - K (z, \zeta)\|_E + \|K (w, z ) - K (\zeta, z) \|_E \le \frac{C \varrho (w, \zeta)^{\beta}}{\varrho (z, \zeta)^{n+1 + \alpha + \beta}},
\eeq
for $z, w, \zeta \in \mathbb{B}_n$ so that $\varrho (z, \zeta ) \ge \delta \varrho (w, \zeta),$ with some (fixed) $\alpha > -1, \delta > 0,$ and $0 < \beta \le 1.$ That is, $K$ is $(n+1+ \alpha)$-dimensional Calder\'{o}n-Zygmund kernel $K$ with values in $E$ on the homogeneous space $(\mathbb{B}_n, \varrho, v_{\alpha}).$

Once the kernel has been defined, then a $\alpha$-time Bergman singular integral operator $T$ is defined as a Calder\'{o}n-Zygmund singular integral operator with a vector-valued kernel $K$ by
\beq\label{eq:BergmanIntOper}
T f (z) = \int_{\mathbb{B}_n} K (z, w) f (w) d v_{\alpha} (w),\quad z, w \in \mathbb{B}_n.
\eeq
If $T$ is bounded from $L^p_{\alpha}$ into $L^p (\mathbb{B}_n, v_{\alpha}; E)$ for any $1 < p <\8,$ we call it a $\alpha$-time Bergman integral operator ($\mathrm{BIO}$). We denote by $\mathrm{BIO}_{\alpha} (E)$ all such operators. If $E=\mathbb{C}$ we write $\mathrm{BIO}_{\alpha} (\mathbb{C}) = \mathrm{BIO}_{\alpha}.$

The examples that we keep in mind are the Bergman projection operator $P_{\alpha}$ from $L^2_{\alpha}$ onto ${\mathcal A}^2_{\alpha},$ which can be expressed as
\be
P_{\alpha}f(z)=\int_{\mathbb{B}_n}K_{\alpha}(z,w)f(w)dv_{\alpha}(w), \quad \forall f\in{L^1(\mathbb{B}_n,dv_{\alpha})},
\ee
where
\beq\label{eq:BergmanKer}
K_{\alpha}(z,w)=\frac{1}{\big(1-\langle z,w\rangle\big)^{n+1+\alpha}},\quad z,w\in\mathbb{B}_n \; \text{with}\; \alpha > -1.
\eeq
Indeed, we have

\begin{proposition}\label{prop:CKestimate}\; (Proposition 2.13 in \cite{T2008})
\begin{enumerate}[{\rm (i)}]

\item there exists a constant $C_1>0$ such that
\be
| K_{\alpha} (z,w) | \le \frac{C_1}{\varrho (z,w)^{n+1+\alpha}},\quad \forall z,w \in \mathbb{B}_n.
\ee

\item There are two constants $C_2, C_3 >0$ such that for all $z,w, \zeta \in \mathbb{B}_n$ satisfying
\be
\varrho (z, \zeta) > C_2 \varrho (w, \zeta)
\ee
one has
\be
| K_{\alpha} (z,w) - K_{\alpha} (z, \zeta)| \le C_3\frac{\varrho (w, \zeta)^{\frac{1}{2}}}{\varrho (z,\zeta)^{n+1+\alpha + \frac{1}{2}}}.
\ee

\end{enumerate}

\end{proposition}

It is well known that $P_{\alpha}$ extends to a bounded operator on $L^p_{\alpha}$ for $1 < p < \8$ (e.g., Theorem 2.11 in \cite{Z}). Thus, we have $P_{\alpha} \in \mathrm{BIO}_{\alpha}.$ This fact will be also concluded from the following result, which is clearly a special case of Proposition \ref{prop:VCZO}.

\begin{theorem}\label{th:VBIO}
Let $E$ be a Banach space and $\alpha > -1.$ Suppose $T$ is a Calder\'{o}n-Zygmund singular integral operator associated with a kernel satisfying \eqref{eq:BergmanKernelEst1} and \eqref{eq:BergmanKernelEst2}. If $T$ is bounded on $L^q (v_{\alpha})$ for some fixed $1< q < \8,$ then $T$ is bounded from $L^p (\mathbb{B}_n, v_{\alpha})$ into $L^p (\mathbb{B}_n, v_{\alpha}; E)$ for every $1 < p < \8,$ and is of weak type $(1, 1).$
\end{theorem}

\subsection{Area functions as vector-valued Bergman integral operators}

Given $\gamma > 0$ and $1 < q < \8.$ Let $E = L^q (\mathbb{B}_n, \chi_{D(0, \gamma)} d \tau).$ We consider the operator
\beq\label{eq:BIOtent}
[T_{\mathrm{tent}} f (z)] (w) = \int_{\mathbb{B}_n} \frac{ f (u) d v_{\alpha} (u)}{(1- \langle \varphi_z (w), u \rangle)^{n+1+\alpha}},\quad \forall z, w \in \mathbb{B}_n,
\eeq
with the kernel
\be
K_{\mathrm{tent}} (z,u) (w) = \frac{ 1}{(1- \langle \varphi_z (w), u \rangle)^{n+1+\alpha}}.
\ee
By the reproduce kernel formula (e.g., Theorem 2.2 in \cite{Z}) we have
\be
[T_{\mathrm{tent}} f (z)](w) = f (\varphi_z (w)), \quad \forall f \in \mathcal{H} (\mathbb{B}_n),
\ee
and hence
\be
\| T_{\mathrm{tent}} f (z) \|_E = A^{(q)}_{\g} (f) (z),
\ee
for any $f \in \mathcal{H} (\mathbb{B}_n).$

\begin{theorem}\label{th:TentBIO}
Let $\g >0, \alpha > -1, 1 < q < \8,$ and $1 < p < \8.$ Then $T_{\mathrm{tent}} \in \mathrm{BIO}_{\alpha} (E).$ Consequently,
\be
\| A^{(q)}_{\g} (f) \|_{L^p_{\alpha}} \lesssim \| f \|_{L^p_{\alpha}},\quad \forall f \in \mathcal{A}^p_{\alpha} (\mathbb{B}_n).
\ee
\end{theorem}

\begin{proof}
Let $f \in L^q_{\alpha} (\mathbb{B}_n).$ Then
\be\begin{split}
\|T_{\mathrm{tent}} f \|^q_{L^q (v_{\alpha}, E)} &= \int_{\mathbb{B}_n} \int_{\mathbb{B}_n} \left | \int_{\mathbb{B}_n} \frac{ f (u) d v_{\alpha} (u)}{(1- \langle \varphi_z (w), u \rangle)^{n+1+\alpha}} \right |^q \chi_{D(0, \gamma)} (w) d \tau (w) d v_{\alpha} (z)\\
& = \int_{\mathbb{B}_n} \int_{\mathbb{B}_n} \left | P_{\alpha} f (w) \right |^q \chi_{D(z, \gamma)} (w) d \tau (w) d v_{\alpha} (z)\\
& \approx \| P_{\alpha} f \|^q_{L^q_{\alpha}} \lesssim \| f \|^q_{L^q_{\alpha}}
\end{split}\ee
by the $L^q$-boundedness of $P_{\alpha}$ for $1 < q < \8.$ This concludes that $T_{\mathrm{tent}}$ is bounded from $L^q_{\alpha}$ into $L^q ( v_{\alpha}, E).$

By Theorem \ref{th:VBIO}, it remains to show that $K_{\mathrm{tent}}$ satisfies the conditions \eqref{eq:BergmanKernelEst1} and \eqref{eq:BergmanKernelEst2}. It is easy to check the condition \eqref{eq:BergmanKernelEst1}. Indeed, by Lemmas \ref{le:Estimation1} and \ref{le:Estimation4} and Proposition \ref{prop:CKestimate} (i) we have
\be\begin{split}
\| K_{\mathrm{tent}} (z,u) \|_E & = \Big ( \int_{\mathbb{B}_n} \frac{ 1}{|1- \langle \varphi_z (w), u \rangle |^{2(n+1+\alpha)}} \chi_{D(0, \gamma)} (w) d \tau (w)\Big )^{\frac{1}{2}}\\
& = \Big ( \int_{D(z, \gamma)} \frac{ 1}{|1- \langle w, u \rangle |^{2(n+1+\alpha)}} d \tau (w)\Big )^{\frac{1}{2}}\\
& \lesssim \frac{ 1}{|1- \langle z, u \rangle |^{n+1+\alpha}}\\
& \le \frac{C}{\varrho (z, u)^{n+1+\alpha}}.
\end{split}\ee
This concludes that $K_{\mathrm{tent}}$ satisfies \eqref{eq:BergmanKernelEst1}.

To check the condition \eqref{eq:BergmanKernelEst2}, we need the following variant of Proposition \ref{prop:CKestimate} (ii).

\begin{lemma}\label{le:BergmanKernelEst}
There exist two constants $C_1, C_2 >0$ such that for all $z, u, \zeta \in \mathbb{B}_n$ satisfying
\be
\varrho (z, \zeta) > C_1 \varrho (u, \zeta)
\ee
one has
\be
| K_{\alpha} (w,u) - K_{\alpha} (w, \zeta)| \le C_2\frac{\varrho (u, \zeta)^{\frac{1}{2}}}{\varrho (z,\zeta)^{n+1+\alpha + \frac{1}{2}}},
\ee
for all $w \in D(z, \g).$
\end{lemma}

The proof can be obtained by slightly modifying the proof of Proposition 2.13 (2) in \cite{T2008} with the help of Lemmas \ref{le:Estimation2} and \ref{le:Estimation4}. We omit the details.

Now we turn out to proceed our proof. Suppose $z,u, \zeta \in \mathbb{B}_n.$ Note that
\be\begin{split}
\| K_{\mathrm{tent}} (z,& u) - K_{\mathrm{tent}} (z,\zeta)\|_E\\
= & \Big ( \int_{\mathbb{B}_n} \Big | \frac{ 1}{(1- \langle \varphi_z (w), u \rangle )^{n+1+\alpha}} - \frac{ 1}{(1- \langle \varphi_z (w), \zeta \rangle )^{n+1+\alpha}} \Big |^2 \chi_{D(0, \gamma)} (w) d \tau (w)\Big )^{\frac{1}{2}}\\
= & \Big ( \int_{\chi_{D(z, \gamma)}} \Big | \frac{ 1}{(1- \langle w, u \rangle )^{n+1+\alpha}} - \frac{ 1}{(1- \langle w, \zeta \rangle )^{n+1+\alpha}} \Big |^2 d \tau (w) \Big )^{\frac{1}{2}}\\
= &  \Big ( \int_{\chi_{D(z, \gamma)}} \Big | K_{\alpha} (w,u) - K_{\alpha} (w, \zeta) \Big |^2 d \tau (w) \Big )^{\frac{1}{2}}.
\end{split}\ee
Then, by Lemma \ref{le:BergmanKernelEst} there exist two constants $C_1, C_2 >0$ such that for all $z, u, \zeta \in \mathbb{B}_n,$
\be
\| K_{\mathrm{tent}} (z,u) - K_{\mathrm{tent}} (z,\zeta)\|_E \le C_2 \frac{\varrho (u, \zeta)^{\frac{1}{2}}}{\varrho (z,\zeta)^{n+1+\alpha + \frac{1}{2}}}
\ee
whenever $\varrho (z, \zeta) > C_1 \varrho (u, \zeta).$

On the other hand, since
\be\begin{split}
\| K_{\mathrm{tent}} (u, & z) - K_{\mathrm{tent}} (\zeta, z)\|_E\\
= & \Big ( \int_{\mathbb{B}_n} \Big | \frac{ 1}{(1- \langle \varphi_u (w), z \rangle )^{n+1+\alpha}} - \frac{ 1}{(1- \langle \varphi_{\zeta} (w), z \rangle )^{n+1+\alpha}} \Big |^2 \chi_{D(0, \gamma)} (w) d \tau (w)\Big )^{\frac{1}{2}}\\
= & \Big ( \int_{\mathbb{B}_n} \Big | K_{\alpha} ( \varphi_u (w), z) - K_{\alpha} ( \varphi_{\zeta} (w), z) \Big |^2 \chi_{D(0, \gamma)} (w) d \tau (w) \Big )^{\frac{1}{2}}\\
\end{split}\ee
and
\be
\varrho (\varphi_u (w), \varphi_{\zeta} (w)) \lesssim | 1- \langle \varphi_u (w), \varphi_{\zeta} (w) \rangle | \thickapprox | 1- \langle u, \zeta \rangle |,\quad \forall w \in D(0, \gamma),
\ee
by Lemma \ref{le:Estimation4} and the inequality
\be
\varrho (z, w) \lesssim | 1 - \langle z, w \rangle |
\ee
(e.g., Eq.(6) in \cite{T2008}), then by slightly modifying the proof of Proposition 2.13 (2) in \cite{T2008} we can prove that
\be
\| K_{\mathrm{tent}} (u, z) - K_{\mathrm{tent}} (\zeta, z)\|_E \lesssim \frac{\varrho (u, \zeta)^{\frac{1}{2}}}{\varrho (z,\zeta)^{n+1+\alpha + \frac{1}{2}}}.
\ee
The details are left to readers. This completes the proof.
\end{proof}

Evidently, we can define:
\begin{enumerate}[{\rm (i)}]

\item The radial area integral operator
\beq\label{eq:BIOradial}
[T_{\mathrm{radial}} f (z)](w) = \int_{\mathbb{B}_n} K_{\mathrm{radial}} (z,u) (w) f (u) d v_{\alpha} (u) ,\quad \forall z, w \in \mathbb{B}_n,
\eeq
with the Bergman kernel
\be
K_{\mathrm{radial}} (z,u) (w) = (n+1+ \alpha) \frac{(1- |\varphi_z (w)|^2) \langle \varphi_z (w), u \rangle}{(1 - \langle \varphi_z (w), u \rangle )^{n+2+ \alpha}},\quad \forall z, u, w \in \mathbb{B}_n.
\ee
It is easy to check that
\be
[T_{\mathrm{radial}} f (z)](w) = (1- |\varphi_z (w)|^2) \mathcal{R} f (\varphi_z (w))
\ee
for any $f \in \H (\mathbb{B}_n).$

\item The complex gradient area integral operator
\beq\label{eq:BIOradial}
[T_{\mathrm{grad}} f (z)](w) = \int_{\mathbb{B}_n} K_{\mathrm{grad}} (z,u) (w) f (u) d v_{\alpha} (u) ,\quad \forall z, w \in \mathbb{B}_n,
\eeq
with the Bergman kernel
\be
K_{\mathrm{grad}} (z,u) (w) = \frac{(n+1+ \alpha) (1- |\varphi_z (w)|^2) \bar{u} }{(1 - \langle \varphi_z (w), u \rangle )^{n+2+ \alpha}},\quad \forall z, u, w \in \mathbb{B}_n.
\ee
It is easy to check that
\be
[T_{\mathrm{grad}} f (z)](w) = (1- |\varphi_z (w)|^2) \nabla f (\varphi_z (w))
\ee
for any $f \in \H (\mathbb{B}_n).$

\item The invariant gradient area integral operator
\beq\label{eq:BIOradial}
[T_{\mathrm{invgrad}} f (z)](w) = \int_{\mathbb{B}_n} K_{\mathrm{invgrad}} (z,u) (w) f (u) d v_{\alpha} (u),\quad \forall z, w \in \mathbb{B}_n,
\eeq
with the Bergman kernel
\be
K_{\mathrm{invgrad}} (z,u) (w) = (n+1+\alpha)\frac{(1-|\varphi_{z}(\omega)|^2)^{n+1+\alpha}\overline{\varphi_{\varphi_{z}(\omega)}(u)}}{|1-\langle \varphi_{z}(\omega), u \rangle|^{2(n+1+\alpha)}},\quad \forall z, u, w \in \mathbb{B}_n.
\ee
It is easy to check that
\be
[T_{\mathrm{invgrad}} f (z)](w) = \tilde{\nabla} f (\varphi_z (w))
\ee
for any $f \in \H (\mathbb{B}_n).$

\end{enumerate}

Similarly, we have

\begin{theorem}\label{th:AreaBIO}
Let $\g >0, 1 < q < \8,$ and $\alpha > -1.$ Then $T_{\mathrm{radial}}, T_{\mathrm{grad}},$ and $T_{\mathrm{invgrad}}$ are all in $\mathrm{BIO}_{\alpha} (E).$ Consequently, $A_{\g}^{(q)} (\mathcal{R} f), A_{\g}^{(q)} (\nabla f),$ and $A_{\g}^{(q)} (\tilde{\nabla} f)$ are all bounded on $\A^p_{\alpha}$ for every $1 < p < \8.$
\end{theorem}

The proof is the same as that of Theorem \ref{th:TentBIO} and the details are omitted.

\end{document}